 \documentclass[11pt,letterpaper]{amsart}

\usepackage{amsfonts,amsmath,latexsym,amssymb,verbatim,amsbsy,amsthm}
\usepackage{dsfont,bm}

\usepackage[dvipsnames]{xcolor}

\usepackage[colorlinks=true, pdfstartview=FitV, linkcolor=RoyalBlue,citecolor=ForestGreen, urlcolor=blue]{hyperref}

\definecolor{labelkey}{rgb}{0,0,1}

\theoremstyle{plain}
\newtheorem{THEOREM}{Theorem}[section]

\newtheorem{theorem}[THEOREM]{Theorem}
\newtheorem{corollary}[THEOREM]{Corollary}
\newtheorem{lemma}[THEOREM]{Lemma}
\newtheorem{proposition}[THEOREM]{Proposition}

\theoremstyle{definition}

\theoremstyle{remark}

\newtheorem{remark}[THEOREM]{Remark}

\newtheorem{example}[THEOREM]{Example}

\newtheorem{conjecture}[THEOREM]{Conjecture}

\newcommand{\thm}[1]{Theorem~\ref{#1}}
\newcommand{\lem}[1]{Lemma~\ref{#1}}

\newcommand{\cor}[1]{Corollary~\ref{#1}}
\newcommand{\prop}[1]{Proposition~\ref{#1}}

\newcommand{\sect}[1]{Section~\ref{#1}}

\newcommand{\N}{\ensuremath{\mathbb{N}}}   
\newcommand{\Q}{\ensuremath{\mathbb{Q}}}   
\newcommand{\R}{\ensuremath{\mathbb{R}}}   
\newcommand{\T}{\ensuremath{\mathbb{T}}}   

\def \b {\beta}

\def \d {\delta}
\def \e {\varepsilon}

\def \n {\nabla}

\def \t {\tau}

\def \o {\omega}

\def \L {\Lambda}

\def \O {\Omega}

\def \bq {{\bf q}}

\def \bv {{\bf v}}

\def \bx {{\bf x}}

\def \bzero {{\bf 0}}

\def \bF {{\bf F}}

\def \cC {\mathcal{C}}

\def \cE {\mathcal{E}}

\def \cH {\mathcal{H}}
\def \cI {\mathcal{I}}

\def \cK {\mathcal{K}}

\def \cP {\mathcal{P}}

\def \cV {\mathcal{V}}

\def \loc {\mathrm{loc}}

\DeclareMathOperator{\supp}{supp} %

\DeclareMathOperator{\diver}{div} %
\DeclareMathOperator{\dist}{dist} %
\def \ss {\subset}

\def \GL {Gr\"onwall's Lemma}

\renewcommand{\geq}{\geqslant}

\renewcommand{\leq}{\leqslant}

\def \dr  {\, \mbox{d}r}
\def \ds  {\, \mbox{d}s}

\def \dtau  {\, \mbox{d}\tau}

\def \ddt  {\frac{\mbox{d\,\,}}{\mbox{d}t}}

\begin{document}

\title[]{Generic alignment conjecture for systems of Cucker-Smale type}

\author{Roman Shvydkoy}

\address{851 S Morgan St, M/C 249, Department of Mathematics, Statistics and Computer Science, University of Illinois at Chicago, Chicago, IL 60607}

\email{shvydkoy@uic.edu}

\subjclass{37A60, 92D50}

\date{\today}

\keywords{emergence,  collective behavior, alignment, Cucker-Smale system, Poincare recurrence}

\begin{abstract} The generic alignment conjecture states that for almost every initial data on the torus solutions to the Cucker-Smale system with a strictly local communication align to the common mean velocity. In this note we present a partial resolution of this conjecture using a statistical mechanics approach. First, the conjecture holds in full for the sticky particle model representing, formally, infinitely strong local communication. In the classical case the conjecture is proved when $N$, the number of agents, is equal to $2$. It follows from a more general result stating that for a system of any size for almost every data at least two agents align.

The analysis is extended to the open space $\R^n$ in the presence of confinement and potential interaction forces. In particular, it is shown that almost every non-oscillatory pair of solutions aligns and aggregates in the potential well. 
\end{abstract}

\maketitle 

\section{Introduction}

The problem of emergence refers to appearance of patterns in systems with self-organization governed by local rules of communication, see \cite{ABFHKPPS,MT2014,Sbook} for extensive surveys.  For models incorporating alignment forces, it means convergence to a common velocity vector (or consensus in the interpretation of opinion dynamics),
\begin{equation}\label{}
\max_{i=1,\ldots,N} |v_i(t) - \bar{v}| \to 0, \text{ as } t \to \infty.
\end{equation}
In this note we address the question in the context of the Cucker-Smale system introduced in \cite{CS2007a,CS2007b}
\begin{equation}\label{e:CS}
\begin{split}
\dot{x}_i & = v_i, \quad x_i \in \O \\ 
\dot{v}_i & = \frac1N \sum_{j=1}^N \phi(x_i-x_j)(v_j - v_i ), \quad v_i\in \R^n,
\end{split}
\end{equation}
where $\phi \in C^1$ is a radially decreasing kernel, and $\O$ is an `environment', typically $\R^n$ or a compact manifold. In the open space case $\O = \R^n$ the following criterion was proved in  \cite{CS2007a} and further refined and extended in \cite{HT2008,HL2009,CFRT2010,TT2014,DS2019}.

\begin{theorem}\label{t:CSinto}
Suppose $\phi$ has heavy tail: $\int_0^\infty \phi(r) \dr = \infty$. Then all solutions to \eqref{e:CS} align exponentially fast to the mean velocity $\bar{v} =\frac{1}{N} \sum_{j=1}^N v_j$, while the flock remains bounded  
\[
\max_{i=1,\ldots,N} |v_i - \bar{v}| \leq C e^{-\d t}, \qquad \max_{i,j=1,\ldots,N} |x_i - x_j| \leq \bar{D}.
\]
The heavy-tail condition is sharp. 
\end{theorem}

In what follows we seek to analyze the situation when the kernel is completely local, i.e.
\begin{equation}\label{e:local}
\supp \phi = B_{r_0}(0),  \text{ for } r_0>0.
\end{equation}
In the open space the counterexample to \thm{t:CSinto} is obvious -- two agents initially separated by a distance larger than $r_0$ and sent in the opposite directions will never come to alignment. It is therefore more natural to address the problem either in the presence of a confinement force (or other constraining mechanisms such as mutual attraction) or in the context of a compact environment such as periodic domain $\O = \T^n$. Still, under completely local protocol \eqref{e:local} there always exists a class of solutions that consist of disconnected and   misaligned agents.  For example, on $\T^n$ one can consider a pair of agents $x_1,x_2$ revolving around two parallel geodesics at a distance $>r_0$ with different velocities. In the case of a confinement force $F = - \n U(x)$ agents may satisfy decoupled Hamiltonian systems
\begin{equation}\label{e:Hintro}
\dot{x}_i=v_i, \qquad \dot{v}_i = - \n U(x_i), 
\end{equation}
without communicating. All these examples are exceptional because they are either described by oscillatory dynamics without alignment \eqref{e:Hintro} or as in the case of the torus, fill a set of measure zero in the ensemble space $\T^{nN} \times \R^{nN}$.

Our approach here is largely motivated by the latter example. We adopt the methodology of the statistical mechanics where solutions to \eqref{e:CS} are viewed as trajectories in the ensemble space
\[
\bx = (x_1,\ldots, x_N)\in \T^{nN}, \quad \bv = (v_1, \ldots, v_N) \in \R^{nN}.
\]
Generic long time behavior of the system will be viewed relative to the classical Lebesgue measure on  $\T^{nN} \times \R^{nN}$.  In the example presented above agents do not interact, and therefore they follow a Euclidean line periodized on the torus 
\begin{equation}\label{e:Euclidean}
\bx(t) = \bx_0 + t \bv_0\mod 2\pi, \quad \bv(t) = \bv_0.
\end{equation}
Since they don't fill the torus densely, leaving the open set $\cup_{i,j} \{ |x_i - x_j| <r_0\}$ unfilled, such solutions must have rationally dependent velocity coordinates of $\bv_0$, and hence form  a negligible set of data. The latter is a classical result in ergodic theory due to Kronecker,  see \cite{Oh2022} for a recent overview.
\begin{theorem}[Kronecker]\label{t:K}
The Euclidean line \eqref{e:Euclidean} densely fills a $k$-dimensional subtorus of $\T^{nN}$ where
\[
k = \dim_\Q \sum_{j=1}^{nN} \Q \bv^j_0.
\]
\end{theorem}

Based on these observations we put forward the following conjecture (see also the discussions in \cite{S-hypo,MT2014,Tadmor-notices,DS2019}).

\begin{conjecture}[Generic Alignment Conjecture]
For almost every initial data  $(\bx_0,\bv_0) \in \T^{nN} \times \R^{nN}$ solutions to the system \eqref{e:CS} align
\[
\max_{i,j} |v_i - v_j| \to 0.
\]
\end{conjecture}

The conjecture will be demonstrated in full for the sticky particle model, which formally corresponds to the case when $\phi = \infty $ for $r<r_0$. The character of dynamics here is of course quite different from the soft interaction model as every communication effectively results in reduction of the number of agents. The final formation would consist of a number of non-interacting clusters and the Kronecker criterion applies, see \sect{s:sticky}.  

In general, our approach will be to study volume compression under the solution map $(\bx_0,\bv_0) \to S_t(\bx_0,\bv_0)$.  The basic idea is to view alignment as convergence to the diagonal set $D = \{ \bv: v_1 = \dots = v_N \}$ which has measure zero. So, shrinking of a volume element to $0$ is necessary for alignment outcome. The Euclidean  distance of $\bv$ to $D$ is exactly the classical quadratic variation
\[
\cV_2 = \sum_{i,j=1}^N |v_i - v_j|^2
\]
which satisfies a dissipation law
\begin{equation}\label{e:V2}
\ddt \cV_2 = - 2 \sum_{i,j=1}^N \phi(x_i - x_j) |v_i - v_j|^2.
\end{equation}
So, $\cV_2$ plays the role of entropy -- a measure of disorder of the velocity field of the flock.  Using an argument based on the  Poincare Recurrence Theorem we establish the following alternative: generically either the agents are not interacting at all from the beginning or they interact sufficiently strongly to compress the volume element to zero. Since the non-interacting agents, according to Kronecker, form a negligible set, we conclude that generically at least one pair of agents keeps interacting infinitely many times during the evolution of the ensemble, and hence, will align. This, in particular, leads to the resolution of the conjecture when $N=2$, see \sect{s:CS}.

\begin{theorem}\label{t:main}
For almost every  $(\bx_0,\bv_0) \in \T^{nN} \times \R^{nN}$ the solution contains a pair of agents that align. Consequently, for $N = 2$  the Generic Alignment Conjecture holds. 
\end{theorem}

A similar methodology applies to the open environment $\O = \R^n$  in the presence of confinement or potential interaction forces. The results of \cite{ShuT2019,ShuT2019anti} show that the generic behavior here is oscillatory. In the local case considered here, it results in a possible positive-measure set of clusters which does not align.  A similar alternative, however, holds: generically, outside the natural set of decoupled oscillators \eqref{e:Hintro} the volume element shrinks to zero, see \prop{p:vol0pot}. As a result, in the confinement case any two-agent system will align and congregate generically, at least when the potential is quadratic, see \thm{t:conf2}.  For the model with pair-wise potential interactions  we obtain the same result but for a much more general class of pairs $(\phi,U)$ which includes the 3Zone model  with repulsion - alignment - attraction forces all present, see \thm{t:inter2} and the discussion afterwards. In particular, we find that alignment and convergence of separation $|x_1 -x_2| $ to the potential well occurs generically for every non-oscillatory data. This extends the near-equilibrium result proved in \cite{ShuT2019anti} to a global statement, see \cor{c:3Z}.

\section{Sticky particle model}\label{s:sticky}

We illustrate the validity of the conjecture on a simpler example of the sticky particle model which corresponds to the situation when  communication is extremely strong within its range:
\begin{equation}\label{}
\phi(x) = \left\{ \begin{split}
\infty, &\quad  |x| \leq r_0;\\
0, & \quad  |x| > r_0.
\end{split}\right.
\end{equation}
The sticky particle dynamics obeys the following rules:

-- if two agents $x_i$, $x_j$ approach distance $r_0$, their velocities switch instantaneously to the average $\frac12 (v_i + v_j)$. From that point on the agents are stuck together for rest of the time.

-- more generally, we say that clusters $\cC^1,\ldots, \cC^K$ each consisting of particles $x^k_i$, $i=1,\ldots,I_k$, $k=1,\ldots,K$ collide at time $t^*$ if for all $t<t^*$ one has $|x^k_i - x^l_j| > r_0$ for all $i,j$ and $k\neq l$, and at $t = t^*$ for each pair $k\neq l$ there exists $i,j$ such that $|x^k_i - x^l_j| = r_0$. If clusters have been traveling with velocities $v^1,\ldots, v^K$, respectively, a new cluster is formed at time $t^*$ traveling with velocity
\begin{equation}\label{e:averule}
v = \frac{1}{|I_1| + \dots + |I_K|} \sum_{k=1}^K |I_k| v^k.
\end{equation}

\begin{proposition}\label{p:sticky}
For almost every initial data $(\bx,\bv) \in \T^{nN} \times \R^{nN}$ the sticky particles congregate in a single cluster. Hence, the Generic Alignment Conjecture holds.
\end{proposition}
\begin{proof}
We can actually specify which data results in single clusters: those $\bv_0$'s that have rationally independent coordinates, i.e. there is no $\bq \in \Q^{nN}\backslash \{\bzero\}$ such that $\bq \cdot \bv_0 = 0$.  Note that the rationally dependent data is a countable union of hyperplanes in $\R^{nN}$, so it is Lebesgue-negligible.  

Now, let $\bv_0$ be rationally independent. Suppose it results in more than one clusters, and we consider a time $T$ beyond which no further gluing occurs. Let us consider two distinct ones $\cC^1$, $\cC^2$ traveling with velocities $v^1$ and $v^2$, respectively. According to the averaging rule \eqref{e:averule} both $v^1$ and $v^2$ are rational convex combinations of the original set of velocities
\begin{equation}\label{e:v1v2}
v^1 = \sum_{i: x_i \in \cC^1} q_i v_i, \qquad v^2 = \sum_{j: x_j \in \cC^2} q_j v_j.
\end{equation}
Since $\cC^1$ and $\cC^2$ never come closer than the communication distance $r_0$, neither does any pair of particles in each cluster. Let us fix $x^1 \in \cC^1 $ and $x^2 \in \cC^2$. They travel at constant velocities $v^1$ and $v^2$ respectively, and $|x^1 - x^2| > r_0$ for all $t>T$.  The above implies that the Euclidean line 
\[
x^1(t) - x^2(t) = x^1(T) - x^2(T) + (t-T)(v^1 - v^2)\mod 2\pi
\]
is non-ergodic. By \thm{t:K} this implies that the coordinates of $v^1 - v^2$ are rationally dependent. In view of \eqref{e:v1v2}, $\bv_0$ is rationally dependent, a contradiction.

\begin{remark}
Considering a more general model  with arbitrary distribution of masses $\{m_i\}_{i=1}^N$, the averaging rule changes to
\begin{equation}\label{e:averulem}
v = \frac{1}{M_1 + \dots + M_K} \sum_{k=1}^K M_k v^k,
\end{equation}
where $M_i$ are the masses of the clusters. Proceeding as in the proof of \prop{p:sticky} we obtain a pair of momenta
\begin{equation}\label{e:v1v2m}
v^1 = \frac{1}{\sum_{i: x_i \in \cC^1} m_i} \sum_{i: x_i \in \cC^1} m_i v_i, \qquad v^2 =  \frac{1}{\sum_{j: x_j \in \cC^2} m_j} \sum_{j: x_j \in \cC^2} m_j v_j
\end{equation}
so that the coordinates of $v^1 - v^2$ are rationally dependent. Since there are only finitely many  possible momenta \eqref{e:v1v2m} we conclude that $\bv_0$ belongs to a negligible set of data, and hence  \prop{p:sticky} still holds for this more general model.

\end{remark}

\end{proof}

\section{Cucker-Smale ensamble}\label{s:CS}

Let us consider now the classical system \eqref{e:CS} with smooth kernel $\phi$. The alignment of a solution $(\bx,\bv)$ can be interpreted as convergence of $\bv$ to the diagonal subspace $D = \{ \bv: v_1 = \dots = v_N \}$. It is easy to show that the square-distance of a field $\bv$ to $D$ is given precisely by the 2-variation
\[
(\dist\{\bv, D\})^2 = \sum_{i,j=1}^N |v_i - v_j|^2: = \cV_2.
\]
The 2-variation is a natural entropy of the system which obeys the following law
\begin{equation}\label{e:V2}
\ddt \cV_2 = - 2 \sum_{i,j=1}^N \phi(x_i - x_j) |v_i - v_j|^2.
\end{equation}
Thus, the distance between $\bv$ and $D$ is non-increasing.  

Since $\T^{nN} \times D$ is  a measure-zero set it is expected that the volume element under the dynamics of \eqref{e:CS} will shrink to zero. This is not exactly the ultimate alignment statement but it is a necessary condition for the conjecture to hold.

Let us denote for short $\phi_{ij} =  \phi(x_i - x_j)$. The divergence of the ensemble field given by \eqref{e:CS} is given by  $- \frac{n}{N} \sum_{i\neq j} \phi_{ij}$. So, the Jacobian of the Cucker-Smale transformation $S_t(\bx_0,\bv_0) = (\bx(t),\bv(t))$ is given by
\begin{equation}\label{e:Jacob}
\det \n_{\bx,\bv} S_t = \exp\left\{ - \frac{n}{N} \int_0^t \sum_{i\neq j} \phi_{ij}(s) \ds \right\}.
\end{equation}
Showing that the Cucker-Smale ensemble shrinks volumes to $0$ comes down to proving the following proposition.
\begin{proposition}\label{p:volume}
We have 
\[
\int_0^\infty \sum_{i\neq j} \phi_{ij}(s,\bx_0,\bv_0) \ds = \infty, 
\]
for  a.e.  $(\bx_0,\bv_0) \in \T^{nN} \times \R^{nN}$.
\end{proposition} 
\begin{proof}
With a slight abuse of notation for a set $F \ss \T^{nN} \times \R^{nN}$ we denote by $|F|$ its Lebesgue measure.

Suppose that the conclusion of the proposition is not true, then there exists a subset $F \ss \T^{nN} \times \R^{nN}$ with $|F| >0$ and $M>0$ such that 
\begin{equation}\label{e:M}
\int_0^\infty \sum_{i\neq j} \phi_{ij}(s,\bx_0,\bv_0) \ds \leq M,
\end{equation}
for all $(\bx_0,\bv_0) \in F$. In particular, this implies that for any $\tilde{F} \ss F$ we have
\begin{equation}\label{e:commen}
 |\tilde{F}| \geq |S_t(\tilde{F})| \geq c_M |\tilde{F}|.
\end{equation}
We can also assume without loss of generality that $F$ is bounded, thus there exists $R>0$ such that $F \ss \T^{nN} \times \{ \max_i |v_i| \leq R\} = \O_R$. Note that by the maximum principle, $S_t(F) \ss \O_R$ for all $t>0$ as well. 

We now show that all agents emanating from set $F$ are generically $r_0$-separated. This will be the content of the following two lemmas.

\begin{lemma}\label{l:aux1}
Let us fix a  $\d>0$ and consider the set 
\[
G_\d = \cup_{i\neq j} \{ (\bx,\bv): \ |x_i - x_j| < r_0 - \d \}.
\]
Then $|F \cap G_\d| = 0$.
\end{lemma}
\begin{proof}
Suppose, on the contrary that  $|F \cap G_\d| >0$. According to \eqref{e:commen} for any $\tilde{F}\ss F$ the action of the transformation $S_n(\tilde{F})$, $n\in \N$ scales the measure of $\tilde{F}$ only by a fixed factor. Thus, the Poincare Recurrence Theorem applies (the classical statement asks for a measure preserving transformation, but only uses the fact that if $S_n(A)$'s are disjoint, then $|A|=0$. This clearly holds in our case too, see \cite{nadkarni}). It says that for almost every point $(\bx,\bv)\in F \cap G_\d$ the action $S_n(\bx,\bv)$ returns to $F \cap G_\d$ infinitely many times. Let us fix any such point $(\bx_0,\bv_0)\in F \cap G_\d$.  So, $S_{n_k} (\bx_0,\bv_0) = (\bx(n_k),\bv(n_k)) \in F \cap G_\d$ for some subsequence $n_k$. In particular $(\bx(n_k),\bv(n_k)) \in G_\d$. This means that for every $k$ there exists $i_k\neq j_k$ such that $|x_{i_k}(n_k) - x_{j_k}(n_k)| < r_0 - \d$. Since all the velocities are bounded by $R$, there exists a time interval $I$ independent of $k$ such that $|x_{i_k}( n_k+s) - x_{j_k}( n_k+s)| < r_0 - \d/2$ for all $s\in I$. Since $\phi$ is properly supported on $B_{r_0}$, see \eqref{e:local}, this further implies that  $\phi_{i_kj_k}( n_k+s) >c_0$ for some $c_0>0$ and all $s\in I$. Consequently, $ \sum_{i\neq j} \phi_{ij}(n_k+s) >c_0$ for all $k$ and $s\in I$.  This clearly contradicts \eqref{e:M}.

\end{proof}

\begin{lemma}\label{l:aux2}
For almost any initial condition $(\bx_0,\bv_0)\in F$ we have $|x_i(t) - x_j(t)| \geq r_0$ for all $i\neq j$ and all $t\geq 0$. 
\end{lemma}
\begin{proof}
It follows from \lem{l:aux1} that
\[
| F \backslash \{ (\bx,\bv): \ |x_i - x_j| \geq  r_0 , \forall i\neq j \}| = 0.
\]
By the semigroup property we have the inclusion $S_{t}(F)\ss F$ for any $t\geq 0$. Thus, by countable selection, there exists a subset $\tilde{F} \ss F$, $|\tilde{F}| = |F|$ such that 
\[
S_{t}(\tilde{F}) \ss \{ (\bx,\bv): \ |x_i - x_j| \geq  r_0 , \forall i\neq j \},\quad \forall t\in \Q_+.
\]
By continuity,
\[
S_{t}(\tilde{F}) \ss \{ (\bx,\bv): \ |x_i - x_j| \geq  r_0 , \forall i\neq j \},\quad \forall t\geq 0.
\]
This proves the lemma.
\end{proof}

According to \lem{l:aux2}, since all $\phi_{ij} = 0$ for almost any initial condition in $F$, the trajectories represent  straight Euclidean lines periodized on the torus 
\[
\bx(t) = \bx_0 + t \bv_0 \mod 2\pi
\]
 with all the agents being separated by $r_0$.  This means that the trajectory is not dense on $\T^{nN}$, and hence, by the Kronecker criterion of \thm{t:K}, $\bv_0$ must have $\Q$-dependent coordinates. Such data cannot fill a set of positive measure. Hence, $|F| = 0$.
\end{proof}

To harvest the consequences of just proved proposition let us prove the following lemma.
\begin{lemma}\label{l:individ}
If for some $i\neq j$, we have 
\[
\int_0^\infty \phi_{ij}(s,\bx_0,\bv_0) \ds = \infty, 
\]
then the agents align, $|v_i - v_j| \to 0$.
\end{lemma}
\begin{proof}\label{ }
The lemma follows in two steps. First, we have the following law for the 1-variation (see \cite{Sbook})
\begin{equation}\label{e:V1}
\cV_1 = \sum_{i,j} |v_i - v_j|, \quad \ddt \cV_1 \leq - \frac{1}{N} \sum_{i,j} \phi_{ij}  |v_i - v_j| : = - \cI_1.
\end{equation}
In particular,
\begin{equation}\label{e:I1}
\int_0^\infty  \cI_1(s) \ds <\infty.
\end{equation}

Next, let compute evolution of an individual difference
\[
\ddt |v_i - v_j|^2  = \frac1N \sum_k \phi_{ik} (v_k - v_i )\cdot (v_i - v_j)  - \frac1N \sum_k \phi_{jk} (v_k - v_j )\cdot (v_i - v_j) 
\]
In the first sum we single out the $k=j$ term and in the second $k=i$ term. The two result in
\begin{multline*}
\ddt |v_i - v_j|^2  = -  \frac1N \phi_{ij} |v_i - v_j|^2 +  \frac1N \sum_{k\neq j} \phi_{ik} (v_k - v_i )\cdot (v_i - v_j) \\ - \frac1N \sum_{k\neq i} \phi_{jk} (v_k - v_j )\cdot (v_i - v_j) ,
\end{multline*}
and hence
\[
\ddt |v_i - v_j|  \leq  -  \frac1N \phi_{ij} |v_i - v_j| + 2 \cI_1.
\]
Applying  \GL, and in view of the assumption and \eqref{e:I1},  we arrive at the conclusion.
\end{proof}

In view of \prop{p:volume}, for almost every data there exists $i\neq j$ which fulfills the \lem{l:individ}. Hence, at least two agents must align, and \thm{t:main} is proved.

\subsection{Clustering Conjecture}

For any initial data $(\bx_0,\bv_0) \in \T^{nN} \times \R^{nN}$ the $\o$-limit set of the trajectory, $\o(\bx_0,\bv_0)$, consists of solutions with constant entropy $\cV_2$.  According to \eqref{e:V2} the dissipation must vanish. This causes clustering: if $|x_i - x_j| <r_0$, and hence $\phi_{ij} >0$, then $v_i = v_j$.  So, agents with different velocities must be separated at least by the communication distance $r_0$. Moreover, all the velocities are constant and all agents travel along straight lines $x_i = x_i(0) + t v_i \mod 2\pi$. If we sort all agents into clusters $\cC^1,\ldots, \cC^K$ which collect agents with the same velocities, then for any $x_k \in \cC^k$, $x_l \in \cC^l$, $k
\neq l$, we have $v_k \neq v_l$ and $|x_k - x_l| \geq r_0$. Denoting $v_1,\ldots, v_K$ all the different velocities in this cluster system we conclude by the Kronecker criterion that 
\begin{equation}\label{e:rv}
r_{kl} \cdot (v_k - v_l) = 0,
\end{equation}
for some $r_{kl} \in \Q^n$.  By Galilean invariance we can also mod out the momentum by imposing another equation
\begin{equation}\label{e:mom0}
\sum_{k} |\cC^k| v_k = 0.
\end{equation}
The system \eqref{e:rv} - \eqref{e:mom0} has $\frac12 K(K-1) + n$ equations imposed upon $nK$ unknowns.
 Generically such a system will be overdetermined if $K > 2n$.  So, we conclude that most likely the number of clusters in the limiting state of the system does not exceed $2n$. This is a weaker form of the Generic Alignment Conjecture.

\begin{conjecture}[Clustering Conjecture]
For almost every initial datum  
\[(\bx_0,\bv_0) \in \T^{nN} \times \R^{nN}\]
 the solution to the system \eqref{e:CS} forms at most $2n$ clusters.
\end{conjecture}

\section{Potential forces}

Let us explore application of the method to the Cucker-Smale system on the open environment $\R^{nN} \times \R^{nN}$  with additional potential forces.  We consider two classes of such forces: confinement 
and mutual interactions. Let us start with the first case.

\subsection{Dynamics under confinement force}
The system is given by
\begin{equation}\label{e:CSconf}
\left\{
\begin{split}
\dot{x}_i&=v_i,\\
\dot{v}_i& = \frac1N \sum_{j=1}^N \phi(x_i-x_j)(v_j-v_i) - \n U(x_i), 
\end{split}\right.
\end{equation}
where the potential $U$ is assumed to be radial and confining
\[
U(r) \to \infty, \quad \text{as }\ r \to \infty.
\]
The total energy which given by 
\begin{equation}\label{e:EKP}
\begin{split}
\cE  &= \cK + \cP, \\
\cK = \frac{1}{2N} \sum_{i=1}^N |\bv_i|^2,& \quad   \cP  = \frac{1}{N} \sum_{i=1}^N U(x_i)
\end{split}
\end{equation}
satisfies the same dissipative law
\begin{equation}\label{e:enlaw}
\ddt \cE = - \frac{1}{N^2} \sum_{i,j=1}^N \phi(x_i - x_j) |v_i - v_j|^2.
\end{equation}
Hence, it remains bounded. This implies that in the confinement case the flock remains bounded in both space and velocity directions (although the diameter may heavily depend on $N$). To put it qualitatively, let us fix an energy level $E>0$ and consider the sub-level set
\[
\O(E) = \{ (\bx,\bv)\in \R^{nN} \times \R^{nN}: \cE \leq E\}.
\]
We can see that $\O(E)$ is a bounded subset of the ensemble space and remains invariant under the action of the semigroup $S_t$. Note that potential forces have no impact on the Jacobian of the transformation $S_t$, and the formula \eqref{e:Jacob} remains the same. Consequently, the arguments of \lem{l:aux1} and \ref{l:aux2} apply on every sub-level set $\O(E)$ to show that for almost any initial data $(\bx_0,\bv_0) \in \R^{nN} \times \R^{nN}$ the following dichotomy holds: either $|x_i(t) - x_j(t)| \geq r_0$ for all $i\neq j$ and all $t\geq 0$, or $\int_0^\infty \sum_{i,j} \phi_{ij} \ds = \infty$.  The first alternative may hold for a set of positive measure. Let us denote it
\begin{equation}\label{ }
\cH = \{ (\bx_0,\bv_0)\in \R^{nN} \times \R^{nN} :  |x_i(t) - x_j(t)| \geq r_0 \ \forall i\neq j,\ \forall t\geq 0\}
\end{equation}
Solutions in $\cH$ will not interact through alignment, $\phi_{ij} = 0$, so the dynamics is described by $N$ decoupled Hamiltonian systems:

\begin{equation}\label{e:H}
\left\{
\begin{split}
\dot{x}_i&=v_i,\\
\dot{v}_i& = - \n U(x_i), 
\end{split}\right.
\end{equation}
So, we can state the dichotomy above as follows.
\begin{proposition}\label{p:vol0pot} For almost every  $(\bx_0,\bv_0) \in \R^{nN} \times \R^{nN} \backslash \cH$, we have 
\[
\int_0^\infty \sum_{i\neq j} \phi_{ij}(s,\bx_0,\bv_0) \ds = \infty.
\]
In particular, it holds for almost every initial data for which $|x_i^0 - x_j^0|<r_0$ for at least one pair of $i\neq j$.
\end{proposition}

In terms of the action of $S_t$, \prop{p:vol0pot} can be interpreted more geometrically:  the semigroup $S_t$ acts invariantly and measure preserving on $\cH$, while it contracts volumes to $0$ as $t\to \infty$ on the complement $\R^{nN} \times \R^{nN} \backslash \cH$.
Despite this contractivity, the generic long time behavior is still more complicated than just convergence to the same velocity. We illustrate this by a simple example.

\begin{example}
Suppose we have a three-agent system $x_1, x_2,x_3$, evolving under the quadratic potential $U = \frac12 r^2$. Let $(x_3,v_3)$ belong to a set of high energy data $F$ which oscillate according to 
\[
\dot{x}_3 = v_3, \qquad \dot{v}_3 = -x_3,
\] 
and such that $|x_3(t)| >R\gg 1$ for all $t>0$ and all $(x_3^0,v_3^0)\in F$. Such a set has positive measure $|F|>0$ in $\R^n \times \R^n$, if $n >1$. Let $\phi$ has a communication range $R/2$ and $\phi(r) = \L$ for $r<R/4$, and $\L \gg 1$. Under such strong communication, any small initial data $(x_1^0,x_2^0; v_1^0,v_2^0) \in \R^2 \times \R^2$ with $|(x_1^0,x_2^0; v_1^0,v_2^0)| < \e \ll R/4$ will never leave a $\d$-neighborhood of the origin as it will be confined to the low energy set $\O(\e)$. Hence, the pair $(x_1,x_2)$ will never interact with $x_3$. However, the pair of interactive agents $x_1,x_2$ will undergo constant alignment interaction as $\phi_{12} = 1$. According to \cite{ShuT2019}, $|x_1 - x_2| + |v_1 - v_2| \lesssim e^{-ct}$ for some $c>0$.  Thus, the configuration data $(\bx,\bv)$ starting on the set $F \times B_\e$ of positive measure will never achieve complete alignment.
\end{example}

Building upon this example one can construct systems of any number of agents $N$ which generically achieve aggregation into an arbitrary number of clusters $K \leq  N$, with $K=N$ corresponding to the data in the oscillation set $\cH$.  Since in all these examples there is always a pair of agents that does align, one is naturally led to believe that a similar statement to that of \thm{t:main} can be proved in the confinement case as well.
The one crucial link that is missing, however, is the analogue of \lem{l:individ}. In other words, even knowing that for some pair of agents $\int_0^\infty \phi_{ij}(s) \ds = \infty$ we cannot conclude the alignment due to the lack of  the $\cV_1$-law \eqref{e:V1}. 

Nonetheless, this issue can be circumvented  for a two agent system with quadratic confinement.

\begin{theorem}\label{t:conf2}
Suppose $N = 2$. Suppose that $\phi(r) \geq 0$ is a smooth radially decreasing kernel, and $U(r) = \frac12 r^2$. Then for almost every initial datum  $(\bx_0,\bv_0) \in \R^{2n} \times \R^{2n} \backslash \cH$ the solution to the system \eqref{e:CSconf}  aligns and aggregates
\[
 |x_1 - x_2| + |v_1 - v_2| \to 0.
\]
\end{theorem}
\begin{proof}
Let us denote for short $v_{12} = v_1 - v_2$ and $x_{12} = x_1 - x_2$. We have
\[
\dot{v}_{12} = - \phi_{12} v_{12} - x_{12}, \quad \dot{x}_{12} = v_{12},
\]
and the energy law reads
\[
\ddt( |v_{12}|^2+ |x_{12}|^2) = - 2 \phi_{12} |v_{12}|^2.
\]
We supplement the energy with the communication-weighted cross-product term with a small factor $\e>0$ to be determined later. So, let us define
\[
\chi = \phi_{12} x_{12} \cdot v_{12}.
\]
Then
\[
\dot{\chi} =  2 \phi'_{12} \frac{(x_{12} \cdot v_{12})^2}{|x_{12}|} + \phi_{12} |v_{12}|^2 -  \phi_{12} |x_{12}|^2 - \phi^2_{12} x_{12} \cdot v_{12}.
\]
The first term is negative, so we will drop it. In the last we use that $\phi$ is bounded and apply the generalized Young inequality,
\[
\dot{\chi} \leq c_1 \phi_{12} |v_{12}|^2 - \frac12 \phi_{12} |x_{12}|^2.
\]

Let us consider the modified energy $\tilde{\cE} = |v_{12}|^2+ |x_{12}|^2 + \e \chi$ and note that for $\e>0$ small $\tilde{\cE} \sim \cE$. Combining the above computations, we obtain
\[
\ddt \tilde{\cE} \leq  - \phi_{12} |v_{12}|^2 - \frac{\e}{2} \phi_{12} |x_{12}|^2 \leq -c_2\phi_{12}  \tilde{\cE}  .
\]
The result follows from \prop{p:vol0pot} and \GL.

\end{proof}

\subsection{Dynamics under potential interactions}

Let us now consider the potential interaction force
\begin{equation}\label{e:CS2Z}
\left\{
\begin{split}
\dot{x}_i&=v_i,\\
\dot{v}_i& = \frac1N \sum_{j=1}^N \phi(x_i-x_j)(v_j-v_i) - \frac1N \sum_{j=1}^N \n U(x_i - x_j),
\end{split}\right.
\end{equation}
where as before the potential $U$ is assumed to be radial and attracting at long range
\[
U(r) \to \infty,\quad \text{as }\ r \to \infty.
\]
The total energy which is given by 

\begin{equation}\label{e:EKPp}
\begin{split}
\cE  &= \cK + \cP, \\
\cK = \frac{1}{2N} \sum_{i=1}^N |v_i|^2,& \quad   \cP  = \frac{1}{N^2} \sum_{i,j=1}^N U(x_{ij})
\end{split}
\end{equation}
The total energy satisfies the same law \eqref{e:enlaw}.

Let us note that the dynamics under \eqref{e:CS2Z} is generically unbounded in the $x$-direction thanks to the conservation of momentum and linear transport of the center of mass: 
\begin{equation}\label{e:meanlaws}
\begin{split}
\ddt \bar{v} & = 0,\quad \bar{v} =  \frac{1}{N} \sum_{i=1}^N v_i, \\
\ddt \bar{x} & = \bar{v}, \quad  \bar{x} =  \frac{1}{N} \sum_{i=1}^N x_i.
\end{split}
\end{equation}
So, to make it bounded it is necessary to mode out the momentum by the Galilean invariance thereby restricting the system to the null-space 
\[
\R^{nN}_ 0 \times \R^{nN}_ 0 = \{ (\bx,\bv):  \bar{x} = \bar{v} = 0\}.
\]
The semigroup $S_t$ restricted to this space will leave the subenergy sets $\O(E)$ invariant, and will be confined to a bounded region in $(x,v)$-space depending only on $E$.

The next step is to compute the divergence of the field given by \eqref{e:CS2Z} restricted to $\R^{nN}_ 0 \times \R^{nN}_ 0 $. The following lemma shows that the divergence is the same as in the unrestricted case.
\begin{lemma}\label{ }
Denting by $\bF$ the field defined by \eqref{e:CS2Z} on $\R^{nN}_ 0 \times \R^{nN}_ 0 $ we have
\[
\diver_{\R^{nN}_ 0 \times \R^{nN}_ 0} \bF = - \frac{n}{N} \sum_{i\neq j} \phi_{ij}.
\]
\end{lemma}
\begin{proof}
Since the divergence is independent of coordinate system, the easiest way to see it is to consider the Cartesian system on the first $N-1$ copies of $\R^n$, and consider the last $\R^n$ as a ``slave" space given by
\[
x_N = -x_1 - \dots - x_{N-1}, \quad v_N = -v_1 - \dots - v_{N-1}.
\]
Since the $x$-component is independent of the $v$-component, we can see that we only have to compute the divergence in first $N-1$ coordinates of the $v$-equation. Let us write, for $i\leq N-1$,
\[
\dot{v}_i = \frac1N \sum_{j=1}^{N-1} \phi_{ij} (v_j-v_i) - \frac1N\phi_{iN} \left( \sum_{k=1}^{N-1} v_k - v_i \right)  - \frac1N \sum_{j=1}^N \n U(x_i - x_j),
\]
So, taking divergence with respect to $v_i$, we obtain
\[
-\frac{n}{N} \sum_{j=1,j\neq i}^{N-1} \phi_{ij} - \frac{2n}{N}\phi_{iN}.
\]
Summing up over $i=1,\ldots,N-1$ we obtain
\begin{equation*}\label{}
\begin{split}
-\frac{n}{N} \sum_{j\neq i}^{N-1} \phi_{ij} - \frac{2n}{N}\sum_{i=1}^{N-1}\phi_{iN} & = -\frac{n}{N} \sum_{j\neq i}^{N-1} \phi_{ij} - \frac{n}{N} \sum_{i=1}^{N-1}\phi_{iN} - \frac{n}{N} \sum_{i=1}^{N-1}\phi_{Ni} \\
& = - \frac{n}{N} \sum_{i\neq j} \phi_{ij}.
\end{split}
\end{equation*}
This proves the lemma.
\end{proof}

So, since $\O(E)$ is a bounded set when restricted on $\R^{nN}_ 0 \times \R^{nN}_ 0$, \lem{l:aux2} applies, and hence we obtain exact same statement as \prop{p:vol0pot} on the null space. By Galilean invariance it extends to the full configuration space $\R^{nN} \times \R^{nN}$ as well simply because the projection on the null space would send a set of positive measure to a set of positive measure. So, the set of points for which the conclusion fails must have measure $0$.

Let us note that for the quadratic potential $U = \frac12 r^2$ the dynamics on the null space coincides with  the confinement case. So, in general, the set of oscillations $\cH$ may have a positive measure.  We project, however, that the Generic 2-agent Conjecture will be valid for the system \eqref{e:CS2Z} as well. 

For $N=2$ the result of \thm{t:conf2} can be restated in a much more general form, and it has several interesting consequences that we will discuss below.

\begin{theorem}\label{t:inter2}
Let $N = 2$. Suppose that $\phi$ is a smooth radial kernel, and $U\in W^{2,\infty}_\loc(\R^n)$ is a radial potential $U = U(r) \geq 0$ such that 
\begin{equation}\label{e:phiU}
U'(r) \phi'(r) \leq 0,\quad  \forall r\geq 0,
\end{equation}
and for all $R>0$ there exists $c_R>0$ such that 
\begin{equation}\label{e:U'U}
|U'(r)|^2 \geq c_R |U(r)|, \qquad \forall  r<R.
\end{equation}
Then for almost every initial datum  $(\bx_0,\bv_0) \in \R^{2n} \times \R^{2n} \backslash \cH$ the solution to the system \eqref{e:CSconf}  satisfies
\[
U(x_1 - x_2) + |v_1 - v_2| \to 0.
\]
\end{theorem}
\begin{proof}
The proof is almost identical to that of \thm{t:conf2}. In view of the oddness of $\n U$ we have
\[
\dot{v}_{12} = - \phi_{12} v_{12} - \n U(x_{12}), \quad \dot{x}_{12} = v_{12},
\]
thus,
\[
\ddt( |v_{12}|^2+ U(x_{12}) ) = - 2 \phi_{12} |v_{12}|^2.
\]
Define
\[
\chi = \phi_{12} \n U(x_{12}) \cdot v_{12}.
\]
Then
\begin{multline*}
\dot{\chi} =  2 \phi'(|x_{12}|)  U'(|x_{12}|) \frac{ (x_{12} \cdot v_{12})^2}{|x_{12}|^2}  -  \phi_{12} |\n U(x_{12})|^2 \\- \phi^2_{12} \n U(x_{12}) \cdot v_{12}+  \phi_{12} D^2 U(x_{12}) v_{12} \cdot v_{12}.
\end{multline*}
where we used that  $\n U(x_{12}) = U'(|x_{12}|) \frac{x_{12}}{|x_{12}|}$. In view of \eqref{e:phiU} the first term is non-positive, and can be dropped. The term $- \phi^2_{12} \n U(x_{12}) \cdot v_{12}$ is treated as earlier. And in the last term we simply use the boundedness of $D^2U$ on bounded sets. So, we obtain
\[
\dot{\chi} \leq c \phi_{12} |v_{12}|^2 - \frac12  \phi_{12} |\n U(x_{12})|^2.
\]
So, for any $\e < \e_0$ we have
\[
\ddt( |v_{12}|^2+\e \chi +  U(x_{12}) ) \leq  - c \e \phi_{12} ( |v_{12}|^2 + |\n U(x_{12})|^2).
\]
In view of  \eqref{e:U'U} we have
\[
|v_{12}|^2+\e \chi +  U(x_{12}) \lesssim  |v_{12}|^2 + |\n U(x_{12})|^2,
\]
and so
\[
\ddt( |v_{12}|^2+\e \chi +  U(x_{12}) ) \leq  - c \e \phi_{12} ( |v_{12}|^2+\e \chi +  U(x_{12}) ).
\]
So, by \GL,
\[
\limsup_{t\to \infty} [ |v_{12}|^2+\e \chi +  U(x_{12}) ]  \leq 0, \quad \forall \e<\e_0.
\]
Notice, however, that $\chi$ is uniformly bounded in time thanks to the fact that the diameter remains bounded. So, from the above
\[
\limsup_{t\to \infty} [ |v_{12}|^2+\e \chi +  U(x_{12}) ]  \geq \limsup_{t\to \infty} [ |v_{12}|^2+  U(x_{12}) ]  - c \e.
\]
So,
\[
\limsup_{t\to \infty} [ |v_{12}|^2+  U(x_{12}) ] \leq c \e, \quad \forall \e<\e_0,
\]
and the theorem follows.
\end{proof}

\thm{t:inter2} describes ultimate outcome for solutions in many various scenarios. Let us discuss a few.

To start, let us note that the assumption \eqref{e:U'U} simply means that near any zero point $U(\ell) = 0$ the order of contact is at most quadratic, $U(r) \sim |r-\ell|^\b$, $\b \leq 2$. Otherwise, $U$ has no critical points. 

So, in our first scenario we assume that $U(r) = (r-\ell_0)_+^2$ for $r<\ell_0 + \d$, and otherwise $U$ is monotonely increasing. Hence, \eqref{e:U'U} holds. Let us also assume that $\phi$ is monotonely decreasing to zero on $r\leq r_0$ (here $r_0$ and $\ell_0$ are in no relation to each other).  Then \eqref{e:phiU} also holds.  The result implies that in this situation generically, 
\[
(|x_1- x_2| - \ell_0)_+ + |v_1 - v_2| \to 0.
\]
So, the separation between agents decreases to a value $\leq \ell_0$, and the agents align. Moreover, if $\ell_0 < r_0$, then for a generic solution starting from some time $t^*$ we have $\phi_{12} \geq \d >0$. We can then conclude $1/t$-decay rate of the energy. Indeed, let us modify the proof of \thm{t:inter2} slightly, where we set $\e = \e(t) \leq \e_0$ to be a decreasing function of time with $|\e'| \leq c \e^2$. Then from time $t^*$ we have
\begin{equation*}\label{}
\begin{split}
\ddt( |v_{12}|^2+\e \chi +  U(x_{12}) ) & \leq  - c  \e(t) ( |v_{12}|^2 + |\n U(x_{12})|^2) + \e' \chi  \\
& \leq - c  (\e(t) - \e^2(t)) ( |v_{12}|^2 + |\n U(x_{12})|^2)  \\
&\leq  - \frac12 c  \e(t) (|v_{12}|^2+\e \chi +  U(x_{12}))  .
\end{split}
\end{equation*}
By \GL, we have
\[
 |v_{12}|^2+ U(x_{12}) \leq e^{-c \int_{t^*}^t \e(\t)\dtau} + C \e(t).
 \]
 The optimal rate is obtained with the choice $\e = \frac{1}{c t}$. 
 
 This represents an extension on the results obtained in \cite{ShuT2019anti} for global communication kernels to a completely local case, albeit only for a two agent system. Let us summarize it.
 
 \begin{corollary} Let $N=2$.
Suppose $U$ is an attraction potential described above. Then  for almost every initial datum  $(\bx_0,\bv_0) \in \R^{2n} \times \R^{2n} \backslash \cH$
\[
(|x_1- x_2| - \ell_0)_+ + |v_1 - v_2| \to 0.
\]
Moreover, if $\ell_0 < r_0$, then there exists $C>0$ such that 
\[
(|x_1- x_2| - \ell_0)_+ + |v_1 - v_2| \leq \frac{C}{\sqrt{1+t}}, \quad t\geq 0.
\]
The above results hold generically for any initial data with $|x_1^0 - x_2^0| < r_0$. 
\end{corollary}

It is notable that the condition $|x_1^0 - x_2^0| < r_0$ imposes no restriction on the initial velocities. So, for example, one can send the agents in almost opposite directions with fast momenta, where they surely separate and  travel long distances vastly exceeding their $\phi$-communication zones. Yet, under the attraction force with probability $1$ they still come back and eventually align.

Our second example pertains to the attraction-alignment-repulsion 3Zone model.  Suppose that for some $\ell_0 \leq \ell_1$
\begin{equation}\label{}
U(r) = \left\{
\begin{split}
\text{smooth and decreasing,} & \quad r_0 \leq \frac12\ell_0,\\
 |r-\ell_0|^2,  & \quad \frac12\ell_0 \leq r \leq \ell_0,\\
 0, & \quad  \ell_0 \leq r \leq \ell_1,\\
 |r-\ell_1|^2,  & \quad r \geq \ell_1,
\end{split}\right.
\end{equation}
In the range $r_0 \leq \frac12\ell_0$ we do not specify any exact rule as long as  $U\in C^2$ near the origin as a function on $\R^n$ and $U$ is radially decreasing.  Next, suppose that $\phi'(r) \geq 0$ for $r\leq \ell_0$, and $\phi'(r) \leq 0$ for $r \geq \ell_1$. In particular, $\phi$ can be simply non-increasing everywhere and is a constant in the range $r\leq \ell_0$. Then \thm{t:inter2} applies.

 \begin{corollary}\label{c:3Z} Let $N=2$.
Suppose $U,\phi$ are as described above. Then  for almost every initial datum  $(\bx_0,\bv_0) \in \R^{2n} \times \R^{2n} \backslash \cH$
\[
\dist\{ |x_1- x_2|, [\ell_0,\ell_1]\} + |v_1 - v_2| \to 0.
\]
Furthermore, if $r_0 > \ell_1$, the above convergence comes with the rate
\[
\dist\{ |x_1- x_2|, [\ell_0,\ell_1]\}  + |v_1 - v_2| \leq \frac{C}{\sqrt{1+t}}, \quad t\geq 0.
\]
In particular, the results hold generically for any initial data with $|x_1^0 - x_2^0| < r_0$.

\end{corollary}

A striking application can be seen in the case when $\ell_0 = \ell_1$ and $r_0 > \ell_0$. This corresponds to another situation addressed in \cite{ShuT2019anti} where it was shown that 
\[
||x_1- x_2| - \ell_0| \lesssim \frac{\ln^{1/2}(t)}{t}, \quad   |v_1 - v_2| \lesssim  \frac{1}{\sqrt{t}},
\]
provided initially 
\begin{equation}\label{e:potequi}
||x^0_1- x^0_2| - \ell_0| +  |v^0_1 - v^0_2| <\e,
\end{equation}
and the communication range is infinite, $r_0 = \infty$. We can see that this result holds globally and even for a finite $r_0$ by \cor{c:3Z}. Near the potential equilibrium \eqref{e:potequi}  the condition $|x_1^0 - x_2^0| < r_0$ is satisfied automatically and without any requirement on velocities.

\section{Statements and Declarations}

Data sharing is not applicable to this article as no datasets were generated or analyzed during the current study.

This work was  supported in part by NSF grant  DMS-2107956.


\end{document}